\documentclass[11pt,reqno,tbtags]{amsart}
\usepackage{amssymb}
\usepackage{natbib}
\bibpunct[, ]{[}{]}{;}{n}{,}{,}

\usepackage{epsfig}
\usepackage{latexsym}

\numberwithin{equation}{section}

\allowdisplaybreaks

\newtheorem{theorem}{Theorem}[section]

\newtheorem{conjectures}[theorem]{Conjectures}

\theoremstyle{definition}

\newtheorem{problem}[theorem]{Problem}
\newtheorem{remark}[theorem]{Remark}

\theoremstyle{remark}

\newenvironment{romenumerate}{\begin{enumerate}
 }{\end{enumerate}}

\newcounter{oldenumi}
{\setcounter{oldenumi}{\value{enumi}}
\begin{romenumerate} \setcounter{enumi}{\value{oldenumi}}}
{\end{romenumerate}}

\newcounter{thmenumerate}

\newcounter{xenumerate}   

\newcommand{\refT}[1]{Theorem~\ref{#1}}

\newcommand{\refR}[1]{Remark~\ref{#1}}
\newcommand{\refS}[1]{Section~\ref{#1}}
\newcommand{\refTab}[1]{Table~\ref{#1}}

\newcommand\marginal[1]{\marginpar{\raggedright\parindent=0pt\tiny #1}}

\begingroup
  \divide\count255 by 60
  \multiply\count255 by -60
  \advance\count255 by \time
  \ifnum \count255 < 10 \xdef\klockan{\the\count1.0\the\count255}
  \else\xdef\klockan{\the\count1.\the\count255}\fi
\endgroup

\newcommand\bigpar[1]{\bigl(#1\bigr)}
\newcommand\Bigpar[1]{\Bigl(#1\Bigr)}

\newcommand\lrpar[1]{\left(#1\right)}

\def\rompar(#1){\textup(#1\textup)}    

\newcommand\parfrac[2]{\lrpar{\frac{#1}{#2}}}

\def\xexp(#1){e^{#1}}

\newcommand\upto{\nearrow}

\newcommand\ie{i.e.\spacefactor=1000}
\newcommand\eg{e.g.\spacefactor=1000}

\newcommand{\as}{a.s.\spacefactor=1000}

\newcommand\ii{\mathrm{i}}

\newcommand\eqd{\overset{\mathrm{d}}{=}}

\newcommand\bbR{\mathbb R}

\newcounter{CC}
\newcounter{cc}

\renewcommand\Re{\operatorname{Re}}

\newcommand\E{\operatorname{\mathbb E{}}}

\newcommand\argmax{\operatorname{argmax}}

\newcommand\gf{\varphi}
\newcommand\gam{\gamma}

\newcommand\eps{\varepsilon}

\newcommand\cS{{\mathcal S}}

\newcommand\qq{^{1/2}}
\newcommand\qqq{^{1/3}}

\newcommand\qqw{^{-1/2}}
\newcommand\qqqw{^{-1/3}}

\newcommand\qw{^{-1}}

\newcommand\qqqbw{^{-2/3}}
\newcommand\qqc{^{3/2}}

\renewcommand{\=}{:=}

\newcommand\intoooo{\int_{-\infty}^\infty}

\newcommand\dd{\,\textup{d}}

\newcommand\rhs{right-hand side}

\newcommand\Ai{\mathrm{Ai}}
 
\newcommand\Bi{\mathrm{Bi}}

\newcommand\hatt{\widehat}
\newcommand\intiooioo{\int_{-\ii\infty}^{\ii\infty}}
\newcommand\gsb{b}

\newcommand{\ba}{\begin{array}}
\newcommand{\ea}{\end{array}}
\newcommand{\bit}{\begin{itemize}}
\newcommand{\eit}{\end{itemize}}
\newcommand{\ben}{\begin{enumerate}}
\newcommand{\een}{\end{enumerate}}
\newcommand{\bal}{\begin{align}}
\newcommand{\bals}{\begin{align*}}
\newcommand{\bs}{\begin{skip}}
\newcommand{\eal}{\end{align}}
\newcommand{\eals}{\end{align*}}
\newcommand{\es}{\end{skip}}

\newcommand\REM[1]{{\raggedright\texttt{[#1]}\par\marginal{XXX}}}

\newcommand\urladdrx[1]{{\urladdr{\def~{{\tiny$\sim$}}#1}}}

\begin{document}
\title[Moments of the location of the maximum]
{Moments of the location of the maximum of Brownian motion with parabolic drift}

\date{18 September 2012}

\author{Svante Janson}
\address{Department of Mathematics, Uppsala University, PO Box 480,
SE-751~06 Uppsala, Sweden}
\email{svante.janson@math.uu.se}
\urladdrx{http://www.math.uu.se/~svante/}
\thanks{Partly supported by the Knut and Alice Wallenberg Foundation}



\begin{abstract} 
We derive integral formulas, involving the Airy function, for moments of the
time a two-sided Brownian motion with parabolic drift attains its maximum.
\end{abstract}

\maketitle

\section{Introduction}\label{S:intro}

Let $W(t)$ be a two-sided Brownian motion with $W(0)=0$;
i.e., $(W(t))_{t\ge0}$ and
$(W(-t))_{t\ge0}$ are two independent standard Brownian
motions.
Fix $\gamma>0$, and consider the Brownian motion with parabolic drift
\begin{equation}\label{wg}
  W_\gamma(t)\=W(t)-\gamma t^2.
\end{equation}
We are interested in 
the maximum 
\begin{equation}\label{mg}
   M_\gamma\=\max_{-\infty<t<\infty} W_\gamma(t)
\end{equation}
of $W_\gam$ (which is \as{} finite), and, in particular,  
the \emph{location} of the maximum, which we denote by 
\begin{equation}\label{zg}
V_\gam\=\argmax_t\bigpar{W_\gam(t)};
\end{equation}
in other words, $V_\gam=t \iff W_\gam(t)= M_\gam$. 
(The maximum in \eqref{mg} is \as{} attained at a unique point, 
so $V_\gam$ is well-defined a.s.)

The parameter $\gam$ is just a scale parameter, see \refS{Scale}, so it can
be fixed arbitrarily without loss of generality.

The distribution of $V_\gam$ was called \emph{Chernoff's distribution} by
\citet{GW01} since it apparently first appeared in \citet{Chernoff}.
It has been studied by several  authors;
in particular, \citet{G85,G89} gave a description of the 
distribution 
and \citet{GW01} give more explicit analytical and numerical formulas;
see also \citet{DS85}.
It has many applications in statistics,
see for example \citet{GW01} and the references given there, or, for a more
recent example, \citet{Anevski}.

The purpose of the present paper is to give formulas for the moments of
$V_\gam$ in terms of integrals involving the Airy function
$\Ai(x)$.
(Recall that $\Ai(x)$ satisfies
$\Ai''(x)=x\Ai(x)$ 
and is
up to a constant factor
the unique solution 
that tends to 0 as $x\upto+\infty$.
See further  \cite[10.4]{AS}.)

All odd moments of $V_\gam$ vanish by symmetry, and 
our main result is the following formula for the even moments,
proved in \refS{Spf}.
(The special case $n=2$ is given by \citet{G11}.)
\begin{theorem}\label{T1}
For every even positive integer $n$, there is a polynomial
$p_n$ of degree at most $n/2$ such that
\begin{equation}\label{evk}
  \E V_\gam^n
=\frac{2^{-n/3}\gam^{-2n/3}}{2\pi\ii} 
\intiooioo \frac{p_n(z)}{\Ai(z)^2}{\dd z}
=\frac{2^{-n/3}\gam^{-2n/3}}{2\pi} 
\intoooo \frac{p_n(\ii y)}{\Ai(\ii y)^2}{\dd y}. 
\end{equation}
In particular, the variance of $V_\gam$ is
\begin{equation}\label{ev}
  \E V_\gam^2=
-\frac{2^{-2/3}\gam^{-4/3}}{6\pi\ii} \intiooioo \frac{z}{\Ai(z)^2}{\dd z}
=\frac{2^{-2/3}\gam^{-4/3}}{6\pi\ii} \intoooo \frac{y}{\Ai(\ii y)^2}{\dd y}. 
\end{equation}
\end{theorem}

The integrals are rapidly converging and easily computed numerically by
standard software.

The polynomials $p_n(z)$ can be found explicitly for any given $n$ by the
procedure in \refS{Spf}, but we do not know any general formula.
They are given for small $n$ in \refTab{Tab}.
See further the conjectures and problems in \refS{S?}.

Numerical values of the first ten absolute moments are given by \citet{GW01};
the first four were computed by Groeneboom and Sommeijer (1984, unpublished).
(Their values for the even moments agree with our formula.)

The maximum $M_\gam$ also appears in many applications.
Its distribution is given in \citet{G85,G89,G10}
and \citet{DS85},
see also,
for example,
\citet{B75,B81},
\citet{D74,D89},
\citet{SJ243}.
(\citet{G89} describes even the
joint distribution of the maximum $M_\gam$ and its location $V_\gam$, see also
\citet{DS85}.)
Formulas for the mean are given by  \citet{DS85}, see also
\citet{SJ243}; in particular 
\begin{equation}\label{em}
  \E M_\gam=
-\frac{2^{-2/3}\gam^{-1/3}}{2\pi\ii} \intiooioo \frac{z}{\Ai(z)^2}{\dd z}. 
\end{equation}
Comparing \eqref{em} and \eqref{ev}, we find the simple relation
\cite{G11}
\begin{equation}
  \E V_\gam^2=\frac1{3\gam}\E M_\gam.
\end{equation}
Since $M_\gam=W_{V_\gam}-\gam V_\gam^2$, this implies that, at the maximum
point, 
\begin{align}
    \E W_{V_\gam}=\frac4{3}\E M_\gam
\qquad\text{and}\qquad
  \E \gam V_\gam^2=\frac1{3}\E M_\gam
=   \frac14 \E W_{V_\gam}.
\end{align}
It would be interesting to find a direct proof of these simple relations.

Formulas for second and higher moments of $M_\gam$ are given in
\citet{SJ243}; however, they are more complicated and do not correspond to
the formulas for moments of $V_\gam$ in the present paper.
It would be interesting to find relations between higher moments of $M_\gam$
and moments of $V_\gam$.

Let us finally mention that the random variable $V_\gam$ is the value at a
fixed time (0, to be precise)
of the stationary stochastic process 
\begin{equation}
V_\gam(x)\=\argmax_t\bigpar{W(t)-\gam(t-x)^2}, 
\end{equation}
which is studied by
\citet{G85,G89,G11}. It would be interesting to find formulas for joint
moments of $V_{\gam}(x)$, in particular for the covariances.

\section{Scaling}\label{Scale}

For any $a>0$, $W(at)\eqd a\qq W(t)$ (as processes on
$(-\infty,\infty)$), and thus
  \begin{align}
V_\gam&= a \argmax \bigpar{W(at)-\gamma (at)^2}
\notag\\&
\eqd a\argmax \bigpar{a\qq W(t)-a^2\gamma t^2}
=aV_{a\qqc\gam}.	\label{b1m}
 \end{align}
The parameter $\gam$ is thus just a scale parameter, and it suffices
to consider a single choice of $\gam$. 
Although the choices $\gam=1$ and $\gam=1/2$ seem most natural, we will use
$\gam=1/\sqrt2$ which gives simpler formulas. We thus define
$V\=V_{1/\sqrt2}$ and have, for any $\gam$,
\begin{equation}\label{vgam}
  V_\gam\eqd2\qqqw\gam\qqqbw V.
\end{equation}

Similarly, 
\begin{equation}
  M_\gam\eqd \gam_1\qqq\gam\qqqw M_{\gam_1}
\end{equation}
for any $\gam,\gam_1>0$.

\section{First formulas for moments}

By \citet[Corollary 3.3]{G89}, $V$ has the density
\begin{equation}\label{vf}
  f(x)=\tfrac12g(x)g(-x),
\end{equation}
where $g$ has the Fourier transform, see \cite[(3.8)]{G89}
(this is where our choice $\gam=2\qqw$ is convenient),
\begin{equation}\label{hg}
  \hatt g(t)\=\intoooo e^{\ii tx}g(x)\dd x = \frac{2^{1/2}}{\Ai(\ii t)},
\qquad -\infty <t<\infty . 
\end{equation}
Note that $|\Ai(\ii t)|\to\infty$ rapidly as $t\to\pm\infty$, 
see \cite[10.4.59]{AS} or \cite[(A.3)]{SJ243},
so $\hatt g(t)$
is rapidly decreasing.
In fact, it follows from the precise asymptotic formula
 \cite[10.4.59]{AS} that
 \begin{equation}\label{sw}
   \Ai(x+\ii y)\qw = O\bigpar{e^{-cy\qqc}}
 \end{equation}
for some $c>0$, uniformly for $|x|\le A$ for any fixed $A$ and $|y|\ge1$,
say (to avoid the zeros of $\Ai$). By differentiation (Cauchy's estimates,
see \eg{}  \cite[Theorem 10.25]{Rudin-real}), 
it follows that the same holds for all derivatives
of $\Ai(x+\ii y)\qw$, and thus all derivatives of $\hatt g(t)$ 
decrease rapidly.
In particular, $\hatt g$
belongs to the Schwartz class $\cS$ of rapidly decreasing functions on
$\bbR$; since this class is preserved by the Fourier transform, also
$g\in\cS$
(see \eg{} \cite[Theorem 7.7]{Rudin-FA}).
In particular, $g$ is integrable.

The characteristic function of $V$ is the Fourier transform $\hatt f(t)$, 
and thus by \eqref{vf}, 
\eqref{hg}
and
standard Fourier analysis 
(see \eg{} \cite[Theorems 7.7--7.8]{Rudin-FA}, but note the different
normalization chosen there),
with $\check g(t)\=g(-t)$,  
\begin{equation}\label{chf}
  \begin{split}
  \gf(t)
&=\hatt f(t)
=\frac12\widehat{g\check g}(t)
=\frac12\cdot\frac1{2\pi}\bigpar{\hatt g*\hatt{\check g}}
=\frac1{4\pi}{\hatt g*\check{\hatt g}}
\\
&=\frac1{2\pi}\intoooo\frac{\dd s}{\Ai(\ii(t+s))\Ai(\ii s)}	
=\frac1{2\pi\ii}\intiooioo\frac{\dd z}{\Ai(z+\ii t)\Ai(z)}.	
  \end{split}
\end{equation}
This is also given in \citet[Lemma 2.1]{G11} (although with typos in the
formula). 
The last integral is taken along the imaginary axis, but we can change the
path of integration, as long as it passes to the right of the zeros $a_n$ of
the Airy function (which are real and negative), for example a line $\Re
z=\gsb$ with 
$\gsb>a_1=-|a_1|$. 

We pause to note the following.

\begin{theorem}
  \label{Tmgf}
 The moment generating function $\E e^{tV}$  
is an entire function of $t$, and is given by, for any 
complex $t$,
\begin{equation}\label{mgf}
  \E e^{tV}
=\frac1{2\pi\ii}\int_{\gsb-\ii\infty}^{\gsb+\ii\infty}
 \frac{\dd z}{\Ai(z+ t)\Ai(z)}, 
\end{equation}
for any real $\gsb$ with $\gsb>a_1$ and $\gsb+\Re t>a_1$.
\end{theorem}

\begin{proof}
The density function $f(t) =\exp(-t^3/3)$ as
$t\to\pm\infty$ by \cite[Corollary 3.4(iii)]{G89}.
Hence,  $\E e^{tV}<\infty$ for
every real $t$, which implies that $\E e^{tV}$ is an entire function of $t$.
The formula \eqref{mgf} now follows from \eqref{chf}
by analytic continuation (for each fixed $\gsb$).
\end{proof}

By differentiation of \eqref{chf} (or \eqref{mgf}) we obtain, for any $n\ge 0$,
\begin{equation}\label{vak}
  \E V^n = (-\ii)^n \frac{\dd^n}{\dd t^n}\gf(0)
=\frac{1}{2\pi\ii}\intiooioo \frac{\dd^n}{\dd z^n}\parfrac{1}{\Ai(z)}
\frac{\dd z}{\Ai(z)}.
\end{equation}

By integration by parts, this is generalized to:
\begin{theorem}
  For any $j,k\ge0$,
\begin{equation}\label{jut}
  \E V^{j+k} 
=\frac{(-1)^{j} }{2\pi\ii}
\intiooioo \frac{\dd^j}{\dd z^j}\parfrac{1}{\Ai(z)}
\cdot\frac{\dd^k}{\dd z^k}\parfrac{1}{\Ai(z)} {\dd z}.
\end{equation}
\end{theorem}

\begin{proof}
  For $j=0$, this is \eqref{vak}.

If we denote the integral on the \rhs{} of \eqref{jut} by $J(j,k)$, then, for
$j,k\ge0$, integration by parts yields
$A(j,k)=-A(j-1,k+1)$, and the result follows by induction.
\end{proof}

Since $\E V^{j+k}=\E V^{k+j}$,
we see again by symmetry that $\E V^n=0$ when $n$ is odd.
For even $n$, it is natural to take $j=k=n/2$ in \eqref{jut}.
For small $n$, this yields the following 
examples.
First, $n=j=k=0$ yields
\begin{equation}\label{ev0}
1=  \E V^{0} 
=\frac{1 }{2\pi\ii}
\intiooioo \frac{1}{\Ai(z)^2},
\end{equation}
as noted by \citet{DS85};
this is easily verified directly, since $\pi\Bi(z)/\Ai(z)$ is a
primitive function of $1/\Ai^2$, see \cite{AlbrightG86}. 

Next, for $n=2$ and $n=4$ we get
\begin{equation}\label{ev2}
  \begin{split}
  \E V^{2} 
=\frac{-1}{2\pi\ii}
\intiooioo \lrpar{\frac{\dd}{\dd z}\parfrac{1}{\Ai(z)}}^2
{\dd z}
=\frac{-1}{2\pi\ii} \intiooioo \frac{\Ai'(z)^2}{\Ai(z)^4}{\dd z}
  \end{split}
\end{equation}
and
\begin{equation}\label{ev4}
  \begin{split}
  \E V^{4} 
&=\frac{1}{2\pi\ii}
\intiooioo\lrpar{ \frac{\dd^2}{\dd z^2}\parfrac{1}{\Ai(z)}}^2
{\dd z}
\\&
=\frac{1}{2\pi\ii}
\intiooioo\lrpar{-\frac{z}{\Ai(z)}+ \frac{2\Ai'(z)^2}{\Ai(z)^3}}^2
{\dd z}.	
  \end{split}
\end{equation}

\begin{remark}\label{Rdiff}
  Since $\Ai''(z)=z\Ai(z)$, it follows by induction that the $m$:th derivative
$\frac{\dd^m}{\dd z^m}\bigpar{\Ai(z)\qw}$ can be expressed as a linear
  combination (with integer coefficients)
of terms
\begin{equation}\label{rdiff}
  \frac{z^j\Ai'(z)^k}{\Ai(z)^\ell}  
\end{equation}
with $j,k\ge0$, $2j+k\le m$ and $\ell=k+1$.
\end{remark}

\section{Some combinatorics of Airy integrals}\label{Scomb}

Inspired by \refR{Rdiff}, we define in general, for any integers
$j,k,\ell\ge0$ with $\ell>k$, 
\begin{equation}
  I(j,k,\ell)\=
\frac{1}{2\pi\ii} \intiooioo \frac{z^j\Ai'(z)^k}{\Ai(z)^\ell}{\dd z}.
\end{equation}
(The integrand decreases rapidly as $z\to\pm\ii\infty$ because $\ell>k$, 
by \eqref{sw} and $\Ai'(z)/\Ai(z)\sim -z\qq$ as $|z|\to\infty$ 
in any sector $|\arg z|\le\pi-\eps$ with $\eps>0$
\cite[10.4.59 and 10.4.61]{AS};
thus the
integral is absolutely convergent.)
Then, recalling $\Ai''(z)=z\Ai(z)$, for any $j,k\ge0$ and $\ell>k$,
\begin{equation}
  \begin{split}
  0&=
\frac{1}{2\pi\ii} \intiooioo 
 \frac{\dd}{\dd  z}\frac{z^j\Ai'(z)^k}{\Ai(z)^\ell}{\dd z} 
\\&
=j I(j-1,k,\ell)+k I(j+1,k-1,\ell-1)-\ell I(j,k+1,\ell+1),
  \end{split}
\end{equation}
where we for convenience define $I(j,k,\ell)=0$ for $j<0$ or $k<0$.
Consequently, for $j,k\ge0$ and $\ell>k$,
\begin{equation}
  I(j,k+1,\ell+1)=\frac j\ell I(j-1,k,\ell)+\frac k\ell I(j+1,k-1,\ell-1),
\end{equation}
or, for $j\ge0$, and $\ell>k\ge1$,
\begin{equation}\label{e44}
I(j,k,\ell)=
  \begin{cases}
\frac{j}{\ell-1} I(j-1,0,\ell-1),
& k=1,
\\
\frac{j}{\ell-1} I(j-1,k-1,\ell-1)
+\frac{k-1}{\ell-1} I(j+1,k-2,\ell-2),
& k\ge2.	
  \end{cases}
\end{equation}
By repeatedly using this relation, any $I(j,k,\ell)$ may be expressed as a
rational combination of $I(p,0,\ell-k)$ with $0\le p\le j+k/2$.
For example,
\begin{align}
 I(0,1,3)&=0;  \label{i013}\\
 I(0,2,4)&=\tfrac13I(1,0,2);\label{i024}\\
 I(1,2,4)&=\tfrac13I(0,1,3)+\tfrac13I(2,0,2)=\tfrac13I(2,0,2);\label{i124}\\
 I(0,4,6)&=\tfrac35I(1,2,4)=\tfrac15I(2,0,2).\label{i046}
\end{align}

\section{Back to moments}\label{Spf}

\begin{proof}[Proof of \refT{T1}]
By combining 
\eqref{vak} and \refR{Rdiff}, 
we can express any moment  $\E V^n$ as a linear
combination with integer coefficients of terms
$I(j,k,k+2)$ with $2j+k\le n$.
By repeated use of \eqref{e44} (see the comment after it), 
this can be further developed into
a linear combination with rational coefficients
of terms
$I(j,0,2)$ with $0\le j\le n/2$.

For example,  by \eqref{ev2} and \eqref{i024},
\begin{equation}\label{ev2a}
  \E V^2=-I(0,2,4)=-\frac13I(1,0,2)=
-\frac{1}{6\pi\ii} \intiooioo \frac{z}{\Ai(z)^2}{\dd z},
\end{equation}
which yields \eqref{ev} by \eqref{vgam}.

To continue with higher moments we have next, by \eqref{ev4} and
\eqref{i124}--\eqref{i046}, 
\begin{equation}\label{ev4a}
  \begin{split}
\E V^4&=I(2,0,2)-4I(1,2,4)+4I(0,4,6)
\\&	
=\Bigpar{1-\frac43+\frac45}I(2,0,2)=\frac{7}{15}I(2,0,2)
\\&
=\frac{7}{30\pi\ii} \intiooioo \frac{z^2}{\Ai(z)^2}{\dd z}. 
  \end{split}
\end{equation}
Similarly (using Maple),
\begin{equation}\label{ev6a}
  \E V^6
=\frac{1}{2\pi\ii} \intiooioo \frac{(26-31z^3)/21}{\Ai(z)^2}{\dd z}. 
\end{equation}

In general this procedure yields, for some rational numbers $b_{nj}$,
\begin{equation}\label{eva}
  \E V^n
=\sum_{j=0}^{n/2} b_{nj}I(j,0,2)
=\frac{1}{2\pi\ii} \intiooioo \frac{p_n(z)}{\Ai(z)^2}{\dd z}. 
\end{equation}
where $p_n(z)\=\sum_{j=0}^{n/2}b_{nj}z^j$ is a polynomial of degree at most $n/2$.
By \eqref{vgam}, this is equivalent to the more general \eqref{evk}.
\end{proof}

For odd $n$, we already know that $\E V^n=0$, so we are mainly interested in
$p_n$ for even $n$. 
The polynomials $p_0,p_2,p_4,p_6$ are implicit in 
\eqref{ev0}, \eqref{ev2a}, \eqref{ev4a} and \eqref{ev6a}, and some further 
cases (computed with Maple) are given in \refTab{Tab}.

\begin{table}
\begin{align*}
  p_0(z)&=1\\
p_2(z)&=-\tfrac13 z\\
p_4(z)&=\tfrac7{15}z^2\\
p_6(z)&=-\tfrac{31}{21}z^3+\tfrac{26}{21}\\
p_8(z)&=\tfrac{127}{15}z^4-\tfrac{196}9z\\
p_{10}(z)&=-\tfrac{2555}{33}z^5+\tfrac{13160}{33}z^2\\
p_{12}(z)&=
{\tfrac {1414477}{1365}}\,{z}^{6}
-{\tfrac {2419532}{273}}\,{z}^{3}+
{\tfrac {1989472}{1365}}.
\end{align*}
\caption{The polynomials $p_n(z)$ for small even $n$.}
\label{Tab}
\end{table}

\begin{remark}
We can see from \refTab{Tab}  that 
(for these $n$) $p_n(z)$ contains only terms $z^j$ where $j\equiv n/2\pmod3$.
This is easily verified for all even $n$: 
a closer look at the induction in \refR{Rdiff} shows that only terms
\eqref{rdiff} with $2j+k\equiv m\pmod3$ appears,
and the reduction in \eqref{e44} preserves $2j+k\pmod 3$.
\end{remark}

\section{Problems and conjectures}\label{S?}

The proof above yields an algorithm for computing the polynomials $p_n(z)$,
but no simple formula for them. We thus ask the following.

\begin{problem}
Is there an explicit formula for the coefficients $b_{nj}$, and thus for
the polynomials $p_n(z)$? Perhaps a recursion formula?
\end{problem}

We have computed $p_n(z)$ for $1\le n\le100$ by Maple, and based on the results
(see also \refTab{Tab}),
we make the following conjectures.

\begin{conjectures}
  \begin{romenumerate}
  \item \label{codd}
$p_n(z)=0$ for every odd $n$.
\item \label{cexact}
$p_n(z)$ has degree exactly $n/2$; \ie, the coefficient $b_{n,n/2}$ of
  $z^{n/2}$ is non-zero.
\item \label{csinh}
These leading coefficients have exponential generating function
\begin{equation}\label{gf}
  \sum_{n=0}^\infty \frac{b_{n,n/2}}{ n!}x^n
=
\frac{x}{\sinh x}.
\end{equation}
  \end{romenumerate}
\end{conjectures}

Of these conjectures, \ref{codd} is natural, since we know that $\E V^n=0$
for odd $n$, and \ref{cexact} is not surprising. 
The  precise conjecture \eqref{gf} is perhaps more surprising. We
have verified that the coefficients up to $x^{100}$ agree, but we have no
general proof.

The simple form of \eqref{gf} suggests also the following open problem.

\begin{problem}
Is there an explicit formula for the generating function 
$\sum_{n=0}^\infty  p_n(z)x^n$? 
\end{problem}

\newcommand\AAP{\emph{Adv. Appl. Probab.} }
\newcommand\JAP{\emph{J. Appl. Probab.} }
\newcommand\JAMS{\emph{J. \AMS} }
\newcommand\MAMS{\emph{Memoirs \AMS} }
\newcommand\PAMS{\emph{Proc. \AMS} }
\newcommand\TAMS{\emph{Trans. \AMS} }
\newcommand\AnnMS{\emph{Ann. Math. Statist.} }
\newcommand\AnnPr{\emph{Ann. Probab.} }
\newcommand\CPC{\emph{Combin. Probab. Comput.} }
\newcommand\JMAA{\emph{J. Math. Anal. Appl.} }
\newcommand\RSA{\emph{Random Struct. Alg.} }
\newcommand\ZW{\emph{Z. Wahrsch. Verw. Gebiete} }
\newcommand\DMTCS{\jour{Discr. Math. Theor. Comput. Sci.} }

\newcommand\AMS{Amer. Math. Soc.}
\newcommand\Springer{Springer-Verlag}
\newcommand\Wiley{Wiley}

\newcommand\vol{\textbf}
\newcommand\jour{\emph}
\newcommand\book{\emph}
\newcommand\inbook{\emph}
\def\no#1#2,{\unskip#2, no. #1,} 
\newcommand\toappear{\unskip, to appear}

\newcommand\webcite[1]{
\texttt{\def~{{\tiny$\sim$}}#1}\hfill\hfill}
\newcommand\webcitesvante{\webcite{http://www.math.uu.se/~svante/papers/}}
\newcommand\arxiv[1]{\webcite{arXiv:#1.}}

\def\nobibitem#1\par{}

\end{document}